\documentclass[11pt,a4paper]{amsart}
\usepackage[english]{babel}
\usepackage{lmodern}
\usepackage{newlfont}
\usepackage{booktabs}
\usepackage{color}
\usepackage[T1]{fontenc}
\usepackage[utf8]{inputenc}
\usepackage{indentfirst}
\usepackage{amsmath}
\usepackage{amsmath,amssymb}
\usepackage{amsthm}
\usepackage{geometry}
\usepackage{mathtools}
\usepackage{listings}
\usepackage{amsfonts}
\usepackage{braket}
\usepackage{emptypage}
\usepackage{newlfont}
\usepackage{amssymb}
\usepackage{graphicx}
\usepackage[italian]{varioref}
\usepackage{accents}
\usepackage{comment}
\usepackage{stmaryrd}
\usepackage{pgfplots}

\makeatletter
\labelformat{equation}{\tagform@{#1}}
\makeatother
\usepackage{hyperref}
\usepackage{relsize}
\usepackage{faktor}
\usepackage{enumerate}
\usepackage{fancyhdr}

\usepackage[colorinlistoftodos]{todonotes}

\DeclarePairedDelimiter{\abs}{\lvert}{\rvert}

\DeclarePairedDelimiter{\scal}{\langle}{\rangle}

\DeclareMathOperator{\hess}{Hess}

\renewcommand{\phi}{\varphi}
\renewcommand{\epsilon}{\varepsilon}

\newcommand{\numberset}{\mathbb}

\newcommand{\N}{\numberset{N}}

\newcommand{\R}{\numberset{R}}

\newcommand{\norm}[1]{\left\|#1\right\|}
\newcommand{\wto}{\rightharpoonup}
\newcommand{\embed}{\hookrightarrow}

\newcommand{\curv}{\mathfrak{K}}
\newcommand{\tangvet}{\textsl{t}}
\newcommand{\orig}{\mathbf{0}}

\newcommand{\la}{\lambda}
\newcommand{\Om}{\Omega}

\theoremstyle{definition}

\theoremstyle{definition}                                                                         
\newtheorem{definizione}{Definizione}[section]
\theoremstyle{definition}                                                                         
\newtheorem{rmk}[definizione]{Remark}
\theoremstyle{plain}                                                                              
\newtheorem{thm}[definizione]{Theorem}
\theoremstyle{plain}    
\newtheorem{prop}[definizione]{Proposition}
\theoremstyle{plain}     
\newtheorem{lemma}[definizione]{Lemma}
\theoremstyle{plain}
\newtheorem{cor}[definizione]{Corollary}
\theoremstyle{definition}

\begin{document}
\title[]{On the number of critical points of the second eigenfunction of the laplacian in convex planar domains}
\thanks{This work was supported by INDAM-GNAMPA}

\author[De Regibus]{Fabio De Regibus}
\address{Dipartimento di Matematica, Universit\`a di Roma ``La Sapienza'', P.le A. Moro 2 - 00185 Roma, Italy, e-mail: {\sf fabio.deregibus@uniroma1.it}.}
\author[Grossi]{Massimo Grossi }
\address{Dipartimento di Matematica, Universit\`a di Roma ``La Sapienza'', P.le A. Moro 2 - 00185 Roma, Italy, e-mail: {\sf massimo.grossi@uniroma1.it}.}

\begin{abstract}
In this paper we consider the second eigenfunction of the Laplacian with Dirichlet boundary conditions in convex domains. If the domain has {\em large eccentricity} then the eigenfunction has {\em exactly} two nondegenerate critical points (of course they are one maximum and one minimum). The proof uses some estimates proved by Jerison (\cite{j}) and Grieser-Jerison (\cite{gj2}) jointly with a topological degree argument. Analogous results for higher order eigenfunctions are proved in rectangular-like domains considered in~\cite{gj4}.
\end{abstract}

\maketitle

\section{Introduction and main results}
Let $\Omega\subset\R^N$, $N\ge2$, be a bounded and smooth domain. Assume that $u$ is a classical solution of the following problem
\begin{equation}
\label{PB}
\begin{cases}
-\Delta u= f(u)&\text{in }\Omega\\
u=0&\text{on }\partial\Omega,
\end{cases}
\end{equation}
where $f:[0,+\infty)\to[0,+\infty)$ is a smooth function.

It is known that the shape of the solution $u$ is strongly influenced by the geometry of the domain $\Omega$ and by the nonlinearity $f$. In particular a classical problem concerns the study of the number of critical points of solutions of problem~\ref{PB}.

If $u$ is a positive solution a lot of results can be found in literature. We are going to recall some of them. The uniqueness of the critical point can be recovered in any dimension and for any locally Lipschitz nonlinearity under symmetry assumptions: this is a consequence of the celebrated results~\cite{gnn} if we ask $\Omega$ to be convex and symmetric with respect to all directions.

Under the only convexity assumption of the domain $\Omega$, the uniqueness of the critical point can be proved only in special cases. If we consider the torsion problem, i.e. $f\equiv1$, Makar-Limanov~\cite{ml} proved uniqueness and nondegeneracy of the critical point when $N=2$. Moreover he showed that $u$ is quasiconcave, that is all the superlevel sets are convex. Then, the same result has been obtained in the case of the first Dirichlet eigenfunction in any dimension, namely $f(u)=\lambda_1u$, see~\cite{bl3,app}. 

In dimension $N=2$, without any symmetry assumption Cabré and Chanillo in~\cite{cc} proved that $u$ possesses exactly one nondegenerate maximum point, provided that the curvature of the boundary of $\Omega$ is strictly positive and $u$ is semi-stable i.e. if for all $\phi\in\mathcal C^{\infty}_{0}(\Omega)$ it holds
\[
\int_{\Omega}|\nabla\phi|^{2}-\int_{\Omega} f'(u)|\phi|^{2}\ge0.
\]
The result has been extended to domain with nonnegative curvature in~\cite{dgm}.

We point out that the convexity assumption can not be dropped, indeed for $N=2$, for any $k\in\N$ it is possible to find a smooth "almost convex" domain $\Omega$ such that the solution of the torsion problem has at least $k$ critical point, see~\cite{ggARXIV} (See also~\cite{drg} for a generalization). 

In this paper we are interested in the study of the number of critical points in the case of sign-changing solutions. To our knowledge there are no results in the literature. So our starting point is the classical problem of the second Dirichlet eigenfunction of the Laplacian in dimension $N=2$, that is we consider the following eigenvalue problem
\begin{equation}
\label{i0}
\begin{cases}
-\Delta u=\lambda_2 u&\text{ in }\Omega\subset\R^2\\
u=0&\text{ on }\partial\Omega,
\end{cases}
\end{equation}
where $\la_2$ is the second eigenvalue of the Laplace operator and $u$ a corresponding eigenfunction. It is known that $u$ must change sign and the geometry and location of its nodal line $\Lambda=\overline{\left\{(x,y)\in\Om:u(x,y)=0\right\}}$ has addressed a lot of interest. A longstanding conjecture is the following
\vskip0.2cm
\hypertarget{(C)}{(C)}\,\,{\em For which domains $\Omega\subset\R^2$ does the nodal line $\Lambda$ touch $\partial\Om$ at exactly two points?}
\vskip0.2cm
In~\cite{pa2} it was conjectured that it happens for any bounded domain and in~\cite{m} it was proved in convex domains, as conjectured in~\cite{ya} (for other works about this conjecture, see for instance~\cite{l2, pa, a,da}). The conjecture is not true in any domain: in~\cite{hhn} it was given an example of a domain with a lot of holes where the nodal line of the second eigenfunction does not touch the boundary. In the same paper it was conjectured that~\hyperlink{(C)}{(C)} holds in planar simply-connected domains.

Of course the computation of the critical points of eigenfunctions to~\ref{i0} is strongly influenced by the geometry of the nodal lines. If it is a closed curve contained in $\Om$  we expect at least $3$ critical points, otherwise $2$ is the minimum number. For this reason we restrict our interest to the case of {\em convex} domains but, even in this case, there are not sufficient qualitative information on the eigenfunction. So we have to consider a suitable subset of convex domains, namely those with {\em large eccentricity}. Let us recall that the eccentricity of a planar domain is defined as
\[
\hbox{ecc}(\Om)=\frac{\hbox{diameter }\Om}{\hbox{inradius }\Om}
\]
where inradius $\Om$ is the radius of the largest circle contained in $\Om$. These domains were considered by Jerison (\cite{j}) and Grieser-Jerison (\cite{gj2}) where the location of the nodal line $\Lambda$ was characterized. In order to state their result we need to normalize the domain $\Om$ in an appropriate way. First let us rotate $\Om$ so that
its projection on the $y$-axis has the shortest possible length, and then dilate
so that this projection has length $1$. Denote by $N$ the length of the projection of $\Om$
on the $x$-axis. Then $N \ge1$, and $N$ is essentially the diameter of $\Om$. From now we denote by $\Om_N$ a domain satisfying the previous properties and accordingly by  $u_N$ a solution to~\ref{i0} in $\Om=\Om_N$ with $\Lambda_N$ its nodal line.

Note that in this setting the domain $\Om_N$ is close to the strip (in a suitable way) $\Om_\infty=\set{(x,y)\in\R^2:0<y<1}$.
We have the following result.
\begin{thm}[{\cite[Theorem 1]{gj2}}]
\label{nodalline}
There is an absolute constat $C_0$ such that the width of the nodal line $\Lambda_N$ is at most $C_0/N$. In other words, up to translate $\Omega_N$, one has
\[
(x,y)\in\Lambda_N\implies \abs{x}<\frac{C_0}{N}.
\]
\end{thm}
This result is our starting point to compute the number of critical points of $u_N$ in $\Om_N$. We have the following theorem.
\begin{thm}
\label{thm1b}
For $N$ large enough, $u_N$ has exactly two critical points  $P_N,Q_N\in\Omega_N$. Moreover $P_N$ (say)  is a nondegenerate maximum point while $Q_N$ is a nondegenerate minimum. Finally $|P_N|,|Q_N|\to+\infty$ as $N\to+\infty$.
\end{thm}

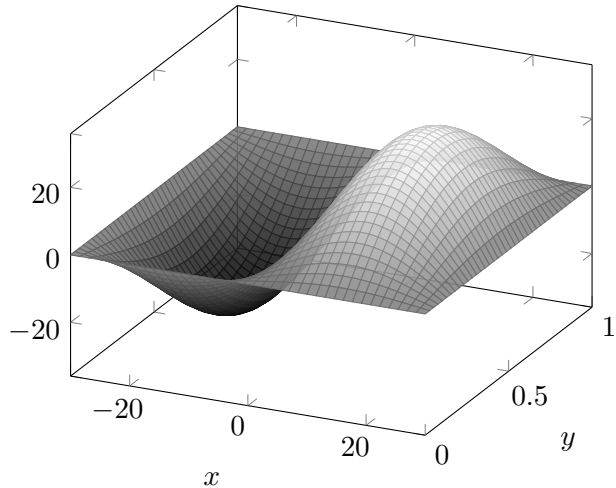
\begin{figure}[h]
\begin{tikzpicture}
\begin{axis}
[xlabel=$x$,
 ylabel=$y$]
\addplot3 [domain=-30:30,
 y domain=0:1, samples=40,surf,
colormap/blackwhite,
opacity=.90
 ]
{30*sin(deg(pi * x / 30)) * sin(deg(pi * y))};
\end{axis}
\end{tikzpicture}
\caption{A graph of $u_N$ for $N$ large.}
\end{figure}

The proof of the previous theorem is splitted in two parts. In the first one we deduce, up to a suitable normalization, the convergence on compact sets of the eigenfuction $u_N$ to the ``limit'' function $u_\infty(x,y)=A_0x\sin(\pi y)$ where $A_0$ is a nonzero constant. This will be done combining some results in~\cite{gj2} and~\cite{gj4}. We stress that the choice of the normalization of the eigenfunction $u_N$ is not a trivial issue, as already discussed in~\cite{j} and~\cite{gj2}.\\
The second part of the proof involves a topological argument: we introduce the vector field $T:\Omega_N\cap\left\{x>\frac12\right\}\to\R^2$
\[
T(q)=(u_{yy}(q)u_{x}(q)-u_{xy}(q)u_{y}(q),u_{xx}(q)u_{y}(q)-u_{xy}(q)u_{x}(q)),\quad q\in\Omega_N\cap\left\{\!x\!>\!\frac12\!\right\},
\]
which allows to ``count'' the critical points of $u_N$. It will be proved that the vector field $T$ is homotopic to the map $I-(x_0,y_0)$ with $(x_0,y_0)\in\Om_N\cap\set{x>\frac12}$ (the same will be done in $\Om_N\cap\set{x<-\frac12}$).
This result, jointly with some properties of the zeros of the vector field $T$,
will give the uniqueness and nondegeneracy of the critical point of $u_N$ in the set where $u_N>0$ and $u_N<0$ respectively.\\
All these computations strongly use the convexity of the domain $\Om_N$ and the convergence of $u_N$ to $u_\infty$. We stress that, although this convergence is only on compact sets, it will be enough to handle the computations in {\em all} $\Om_N$. 

In the last part of the paper we deal with a particular class of convex domain not included in the previous section, which are perturbation of rectangles which still converge to the strip. This family of domains has been studied in~\cite{gj4} where they give a full asymptotic expansion for the $m$-th Dirichlet eigenvalue and for the associated eigenfunction (see Theorem~\ref{thmGJPacific}).

Let $\phi:[0,1]\to[0,\infty)$ be a Lipschitz and concave function and for $N\in[0,\infty)$ set
\begin{equation}
\label{rett}
\mathcal  R_N:=\Set{(x,y)\in\R^2|0<y<1,\ -\phi(y)<x<N}.
\end{equation}

Let $u_{m,N}\in\mathcal C^\infty(\mathcal R_N)$ be the $m$-th Dirichlet eigenfunction in $\mathcal R_N$ which solves
\[
\begin{cases}
-\Delta u_{m,N}=\lambda_{m,N} u_{m,N}&\text{ in }\mathcal R_N\\
u_{m,N}=0&\text{ on }\partial \mathcal R_N.
\end{cases}
\]
where $\lambda_{m,N}$ is the $m$-th eigenvalue.
In next theorem we prove the existence of exactly $m$ critical points for $u_{m,N}$ in $\mathcal  R_N$.
\begin{thm}
\label{thm1}
For $N$ large enough, $u_{m,N}$ has exactly $m$ nondegenerate critical points in the set $\mathcal  R_N$. Moreover all of them are maxima and minima.
\end{thm}
Unlike Theorem~\ref{thm1b}, the proof of Theorem~\ref{thm1} is much easier and it strongly follows by the estimates proved in~\cite{gj4}.\\
The paper is organized as follows: in the next section we recall some notations and some results for the second eigenfunction on convex domain with high eccentricity from the papers of Jerison and Grieser, Jerison and in the next one we extrapolate the local convergence of $u_N$ to $u_\infty$ (see Proposition~\ref{convcompatti}). Section~\ref{sez3} is devoted to the topological argument where we perform the computations involving the vector field $T$ and we prove Theorem~\ref{thm1b}. 
Finally, in the last section we investigate the eigenfunctions on convex perturbations of long rectangles, proving Theorem~\ref{thm1}.

\section{Preliminary results}
\label{sez2}
In this section we collect some results proved in~\cite{j,gj2} (see also~\cite{j2} for an overview of the problem).
As we pointed out in the Introduction, let us rotate $\Omega_N$ so that its projection on the $y$-axis has the shortest possible length, then dilate so that this projection has length $1$. Denote by $N$ the length of the projection of the domain on the $x$-axis, then $N\ge1$. Hence, we write
\[
\Omega_N=\Set{(x,y)\in\R^2|f_{1,N}(x)<y<f_{2,N}(x),\quad x\in(a_N,b_N)},
\]
where $b_N-a_N=N$, $0\le f_{1,N}\le f_{2,N}\le1$, and the height function of $\Omega_N$ is $h_N:=f_{2,N}-f_{1,N}$. We require that
\[
f_{1,N}\to0\hbox{ and }f_{2,N}\to1\hbox{ in }C^\infty_{loc}(\R)\hbox{ as }N\to+\infty.
\]
By the convexity of $\Om_N$ we have that $f_{1,N}''\le0$ and $f_{2,N}''\ge0$.
Our assumptions imply that the set $\Om_N$ ``converges'' to the strip
\[
\Om_\infty:=\Set{(x,y)\in\R^2|0<y<1}.
\]
More precisely we have that for all compact sets $K\subset\R^2$ one has $\abs{(\Omega_N\Delta\Om_\infty)\cap K}\to0$. As we recalled in the Introduction we know that the nodal line
\[
\Lambda_N:=\overline{\Set{(x,y)\in\Omega_N|u_N(x,y)=0}},
\]
is close to the straight line $\{x=0\}$, up to a translation (see Theorem~\ref{nodalline} in the Introduction above). Finally let  $u_N\in\mathcal C^\infty(\Omega_N)$ be the solution of
\begin{equation}
\label{PB1b}
\begin{cases}
-\Delta u=\lambda_{2,N} u&\text{ in }\Omega_N\\
u=0&\text{ on }\partial\Omega_N,
\end{cases}
\end{equation}
and for all $(x_0,y_0)\in\Lambda_N\cap\Omega_N$ we can assume that $u_N(x_0+1,y_0)>0$ and $u_N(x_0-1,y_0)<0$, that is $u_N>0$ on the right of the nodal line and $u_N$ is negative on the left.

Finally, let $L_N$ be the length of the longest interval $I_{L_N}\subset(a_N,b_N)$ such that
\[
h_N(x)=f_{2,N}-f_{1,N}\ge1-\frac{1}{L_N^2},\quad\text{in }I_{L_N}.
\]
The number $L_N$ is related to the length of the rectangle contained in $\Omega_N$ with lowest first eigenvalue and it satisfies the following bounds (see~\cite{gj2,j2})
\begin{equation}
\label{boundL}
N^{1/3}\le L_N\le N.
\end{equation}

For future convenience, we introduce for $k\in\R$ the sets
\[
\Omega_N^k:=\Set{(x,y)\in\Omega_N|-k<x<k},
\]
and
\[
\Omega_\infty^k:=\Set{(x,y)\in\R^2|-k<x<k,\,\,0<y<1},
\]
where we remember that $\Omega_\infty=\R\times(0,1)$ is the infinite strip of height $1$.
Since $0\le f_{1,N}\le f_{2,N}\le1$, we have that the continuous embedding $H^1_0(\Omega_N)\embed H^1_0(\Omega_\infty)$ holds true by means of zero extension outside $\Omega_N$.

An important step to deduce good estimates for the eigenfunction $u_N$ is to choose a correct normalization. So let us define $\widehat u_N$ as
\[
\widehat u_N:=L_N\frac{u_N}{||u_N||_\infty}.
\]
With a little abuse of notation, in the following we will set 
\[
\widehat u_N=u_N.
\]
From the results in~\cite{gj2} we will deduce the following lemma.
\begin{lemma}
\label{pippo}
There exists a positive constant $C$  independent of $N$, such that 
\begin{equation}
\label{stimasublineare}
\abs{u_N(x,y)}\le C(1+\abs{x}),\quad\forall (x,y)\in\Omega_N,
\end{equation}
and
\begin{equation}
\label{stimasuplineare}
\abs{u_N(\pm1,1/2)}\ge \frac{1}{C}.
\end{equation}
\end{lemma}

\begin{proof}
The first estimate~\ref{stimasublineare} is proved in~\cite[Theorem 4]{gj2}.

To prove~\ref{stimasuplineare}, still recalling~\cite{gj2}, define the following function
\[
\tilde u_N(x,y):=\psi_N(x)\sqrt{\frac{2}{h_N(x)}}\sin\left(\pi\frac{y-f_{1,N}(x)}{h_N(x)}\right),
\]
where
\[
\psi_N(x):=\sqrt{\frac{2}{h_N(x)}}\int_{f_{1,N}(x)}^{f_{2,N}(x)}\sin\left(\pi\frac{y-f_{1,N}(x)}{h_N(x)}\right)u_N(x,y)\,dy.
\]
Note that $\tilde u_N(x,y)\sim\sqrt2\sin(\pi y)\psi_N(x)$ and $\psi_N(x)\sim \sqrt2\int_0^1\sin(\pi y)u_N(x,y)\,dy$ if $x$ is bounded.
Finally, let $v_N:=u_N-\tilde u_N$.

Now, let $C>0$ be any positive constant independent from $N$ which may vary in the rest of the proof and recall the following estimates.~\cite[Equation (26)]{gj2} tells us
\[
\abs{\psi_N(x)}\ge C\abs{x},\quad-2<x<2,
\]
and~\cite[Lemma 5]{gj2} gives for all $(x,y)\in\Omega_N^2$
\begin{align*}
&\abs{v_N(x,y)}\\
&\le \sqrt{\!\frac{2}{h_N(x)}}\sin\!\left(\!\pi\frac{y-f_{1,N}(x)}{h_N(x)}\right)\!\left(1+\abs{x}\left|\log\!\left(\!\sqrt{\frac{2}{h_N(x)}}\sin\!\left(\!\pi\frac{y-f_{1,N}(x)}{h_N(x)}\right)\right)\right|\right)\!L_N^{-3}\\
&\le\frac{C}{L_N^3}.
\end{align*}
Hence for $(x,y)\in\Omega_N^2$ one has
\begin{align*}
\abs{u_N(x,y)}&=\abs{\tilde u_N(x,y)+v_N(x,y)}\\
	&\ge\left|\psi_N(x)\sqrt{\frac{2}{h_N(x)}}\sin\left(\pi\frac{y-f_{1,N}(x)}{h_N(x)}\right)\right|-\abs{v_N(x,y)}\\
	&\ge\abs{\psi_N(x)}\sin\left(\pi\frac{y-f_{1,N}(x)}{h_N(x)}\right)-\frac{C}{L_N^3}\\
	&\ge C\abs{x}\sin\left(\pi\frac{y-f_{1,N}(x)}{h_N(x)}\right)-\frac{C}{L_N^3}.
\end{align*}
Finally, since for $N\to+\infty$ from~\ref{boundL} also $L_N\to+\infty$, one has $(\pm1,(f_{1,N}(1)+f_{2,N}(1))/2)\to(\pm1,1/2)$, and then we have
\begin{align*}
\abs{u_N(\pm1,1/2)}&=\abs{u_N(\pm1,(f_{1,N}(1)+f_{2,N}(1))/2)}+o(1)\\
&\ge C\abs{\pm1}(1+o(1))\ge\frac{C}{2}.\qedhere
\end{align*}
\end{proof}

\begin{rmk}
From~\ref{stimasublineare} one has
\begin{equation}
\label{stimacompattik}
\norm{u_N}_{L^\infty(\Omega_\infty^{k})}\le C(1+k),\quad\forall k\in\N.
\end{equation}
\end{rmk}

The following lemma follows by the standard elliptic regularity theory.

\begin{lemma}
\label{reg}
For $m\in\N$, $f\in H^m(\Omega_N^{k+1})$, let $u\in H^1(\Omega_N^{k+1})$ be a weak solution of 
\[
\begin{cases}
-\Delta u=f&\text{in }\Omega_N^{k+1}\\
u=0&\text{on }\partial\Omega_N^{k+1}\setminus\set{x=\pm(k+1)}.
\end{cases}
\]
Then for $\delta\in(0,1)$ it holds
\[
u\in H^{m+2}(\Omega_N^{k+\delta}),
\]
with the estimate
\[
\norm{u}_{H^{m+2}(\Omega_N^{k+\delta})}\le C\left(\norm{f}_{H^m(\Omega_N^{k+1})}+\norm{u}_{L^2(\Omega_N^{k+1})}\right),
\]
for some $C>0$ independent from $N$.
\end{lemma}
We point out that the independence from $N$ follows from the convergence of $\Omega_N$ to $\Omega_\infty$, that is the fact that $\abs{(\Omega_N\Delta\Omega_\infty)\cap K}\to0$, for all compact sets $K\subset\R^2$.

\section{The asymptotic behavior of $u_N$ }
In this section we study the limiting behavior of the solution $u_N$ on compact sets. In particular, $u_N$ converges to a function which is a solution in the whole strip $\Omega_\infty$.
\begin{prop}
\label{convcompatti}
Up to renormalize $u_N$, we have that for all multiindices $\alpha$, with $\abs{\alpha}\le2$ and fixed $k\in\N$, it holds
\begin{equation}
\label{GJ:conv1b}
\sup_{\overline\Omega_N\cap\{-k\le x\le k\}}\left| D^\alpha\big(u_N-A_0x\sin(\pi y)\big)\right|=o(1),\quad\text{for }N\to+\infty,
\end{equation}
for some suitable constant $A_0\ne0$.
\end{prop}
The proof of the previous proposition is a consequence of the next two lemmas.

\begin{lemma}
\label{convcompatti2}
We have that there exists $u_\infty:\Om_\infty\to\R$ such that  for all multiindices $\alpha$, with $\abs{\alpha}\le2$ and fixed $k\in\N$, it holds up to subsequences
\begin{equation}
\label{GJ:conv1b2}
\sup_{\overline\Omega_N\cap\{-k\le x\le k\}}\left| D^\alpha\big(u_N-u_\infty\big)\right|=o(1),\quad\text{for }N\to+\infty,
\end{equation}
and $u_\infty$ solves
\[
\begin{cases}
-\Delta u_\infty=\pi^2u_\infty&\text{in }\Omega_\infty\\
u_\infty=0&\text{for }y=0,1.
\end{cases}
\]
\end{lemma}
\begin{proof}
In the proof of the lemma, convergence will be understood up to subsequences.

Fix $k\in\N$. From~\ref{stimacompattik} and Lemma~\ref{reg} we have
\[
\norm{u_N}_{H^2\left(\Omega_\infty^{k+\frac{1}{2}}\right)}\le C(k),
\]
for some $C(k)>0$ and so there exists $u_\infty^k\in H^1\left(\Omega_\infty^{k+\frac{1}{2}}\right)$ such that
\[
u_N\wto u_\infty^k\quad\hbox{weakly in }H^1\left(\Omega_\infty^{k+\frac{1}{2}}\right).
\]
Let us show that in $\Omega_\infty^{k+\frac{1}{2}}$ we have that
  $-\Delta u_\infty^k=\pi^2u_\infty^k$ in weak sense. Indeed, for all $\phi\in\mathcal C^\infty_0\left(\Omega_\infty^{k+\frac{1}{2}}\right)$ one has
\begin{align*}
\int_{\Omega_\infty^{k+\frac{1}{2}}}\nabla u_\infty^k\nabla\phi&=\int_{\Omega_\infty^{k+\frac{1}{2}}}(\nabla u_\infty^k\nabla\phi+u_\infty^k\phi)-\int_{\Omega_\infty^{k+\frac{1}{2}}}u_\infty^k\phi\\
&=\lim_N\int_{\Omega_\infty^{k+\frac{1}{2}}}(\nabla u_N\nabla\phi+u_N\phi)-\lim_N\int_{\Omega_\infty^{k+\frac{1}{2}}}u_N\phi\\
&=\lim_N\int_{\Omega_\infty^{k+\frac{1}{2}}}\nabla u_N\nabla\phi\\
&=\lim_N\lambda_{2,N}\int_{\Omega_\infty^{k+\frac{1}{2}}}u_N\phi\\
&=\pi^2\int_{\Omega_\infty^{k+\frac{1}{2}}}u_\infty^k\phi.
\end{align*}
Moreover, it is not difficult to see that 
\[
u_\infty^k=0,\quad\text{on }\partial\Omega_\infty^{k+\frac{1}{2}}\setminus\set{x=\pm(k+\frac{1}{2})},
\]
and by Lemma~\ref{reg} we obtain that $u_\infty^k\in\mathcal C^\infty\left(\Omega_\infty^{k+\frac{1}{3}}\right)$.

By~\ref{stimasuplineare} we deduce that $u_\infty^k\not\equiv 0$ in $\Omega_\infty^k$, and from the assumptions on the nodal lines of $u_N$ one has $u_\infty^k(0,y)=0$ for all $y\in(0,1)$ .

Next we show the $\mathcal C^2$ convergence up to the boundary of $\Omega_N^k$. Let us start by fixing a point $(x,0)$ with $-k<x<k$. From the assumption on $\Omega_N$ we can define the set
\[
B(N):=\Omega_N\cap B_r(x,0)=\Set{(x,y)\in B_r(x,0)|y> f_{1,N}(x)},
\]
for some $r>0$ suitably small. Then, from the standard regularity theory we deduce that
\[
\norm{u_N-u_\infty^k}_{C^2(B_{1/2}(N))}\to0,\quad\text{for }N\to+\infty.
\]
 where $B_{1/2}(N):=\Omega_N\cap B_{r/2}(x,0)$. To show $\mathcal C^2$ convergence in the whole $\Omega_\infty^{k}$ it is enough to cover the segments $(-k,k)\times\{0\}$ and $(-k,k)\times\{1\}$ with finitely many balls.\\
 Thus we have proved that for all $k\in\N$ we can find a function $u_\infty^k\in\mathcal C^\infty(\Omega_\infty^k)$ such that $u_N\to u_\infty^k$ in $\mathcal C^2(\Omega_\infty^k)$ and $u_\infty^k$ solves
\[
\begin{cases}
-\Delta u_\infty^k=\pi^2u_\infty^k&\text{in }\Omega_\infty^{k}\\
u_\infty^k=0&\text{for }y=0,1.
\end{cases}
\]
By uniqueness of the limit we have $u_\infty^{k+1}=u_\infty^k$ in $\Omega_\infty^k$, and this allows us to define a $\mathcal C^2$ function in the whole strip $\Omega_\infty$ given by
\[
u_\infty(x,y):=u_\infty^k(x,y),\quad\text{for }(x,y)\in\Omega_\infty^k,
\]
which is a solution of
\begin{equation}
\label{PBlimte}
\begin{cases}
-\Delta u_\infty=\pi^2u_\infty&\text{in }\Omega_\infty\\
u_\infty=0&\text{for }y=0,1.
\end{cases}
\end{equation}
Moreover, from the corresponding properties of $u_\infty^k$, note that $u_\infty(0,y)=0$ for all $y\in(0,1)$ and $\abs{u_\infty(\pm1,1/2)}>0$.
\end{proof}
To conclude the proof of Proposition~\ref{convcompatti} we must prove that $u_\infty(x,y)=A_0x\sin(\pi y)$ for some $A_0>0$. This is a consequence of the next lemma.
\begin{lemma}\label{conv3}
The functions $u(x,y)=Ax\sin(\pi y)$ are the unique solutions of the problem
\begin{equation}\label{c1}
\begin{cases}
-\Delta u=\pi^2u&\text{in }\Omega_\infty\\
u(0,y)=0&\text{for any }y\in[0,1]\\
u(x,0)=u(x,1)=0&\text{for any }x\in\R\\
\abs{u(x,y)}\le C(1+\abs{x})&\hbox{for some constant }C>0,
\end{cases}
\end{equation}
for any $A\in\R$.
\end{lemma}
\begin{proof}
Here we follow~\cite[Lemma 6]{gj4}.
Let $u(x,y)$ be a solution to~\ref{c1}. Then for each fixed $x$ its Fourier series is given by
\[
u(x,y)=\sum_{j=1}^\infty A_j(x)\sin(j\pi y),
\]
where
\begin{equation}
\label{defAj}
A_j(x):=2\int_0^1u(x,t)\sin(j\pi t)\,dt,
\end{equation}
that is $A_1(x)=c_1x+d_1$ and
\[
A_j(x)=c_je^{-\sqrt{j^2-1}\pi x}+d_je^{\sqrt{j^2-1}\pi x},\quad\text{for }j\ge2,
\]
with $c_j,d_j\in\R$ for all $j\ge1$, see~\cite[Lemma 6]{gj4} for more details.

Then we evaluate~\ref{defAj} for $x=0$ and taking into account that $u(0,y)=0$ for all $y\in[0,1]$ we have
\[
d_1=A_1(0)=2\int_0^1u(0,y)\sin(\pi y)\,dy=0,
\]
and
\begin{equation}
\label{cj}
c_j+d_j=A_j(0)=2\int_0^1u(0,y)\sin(j\pi y)\,dy=0,
\end{equation}
for $j\ge2$.

By the definition of $A_j(x)$ and since $u$ has growth at most linear we have that $d_j=0$ for all $j\ge2$. Hence~\ref{cj} implies $c_j=0$ for all $j\ge2$ and then
\[
u(x,y)=\sum_{j=1}^\infty A_j(x)\sin(j\pi y)=A_1(x)\sin(\pi y)=(c_1x+d_1)\sin(\pi y)=c_1x\sin(\pi y),
\]
and the claim follows.
\end{proof}
Now we are in the position to give the proof of Proposition~\ref{convcompatti}.
\begin{proof}[Proof of Proposition \ref{convcompatti}]
By Lemma~\ref{convcompatti2} $u_N$ converges up to a subsequence to $u_\infty$, let us show that $u_\infty(x,y)=A_0x\sin(\pi y)$. First we observe that from inequality~\ref{stimasublineare} in Lemma~\ref{pippo} we know that $u_\infty$ has growth at most linear for $x\to\pm\infty$. Hence Lemma~\ref{conv3} applies and so $u_\infty(x,y)=Ax\sin(\pi y)$. Finally $A=A_0=u_\infty(1,1/2)>0$. To conclude the proof we need to show that, up to renormalize some $u_N$ the convergence holds for the whole sequence. By contradiction, assume that we can find a subsequence $(u_{N_m})_m\subset(u_{N})_N$ not converging to $u_\infty$ and $C>0$ such that
\[
\norm{u_{N_m}-A_0x\sin(\pi y)}_{L^\infty(\Omega_{N_m}\cap\{-k< x< k\})}\ge C.
\]
Now, we can apply Lemma~\ref{convcompatti2}, and in turn Lemma~\ref{conv3}, to the sequence $(u_{N_m})_m$ to find that, up to subsequences
\[
\norm{u_{N_m}-A_1x\sin(\pi y)}_{L^\infty(\Omega_{N_m}\cap\{-k< x< k\})}\to 0,\quad\text{for }m\to+\infty,
\]
for some $A_1>0$. Hence, up to multiply $u_{N_m}$ by $A_0/A_1$ we get $u_{N_m}\to u_\infty$, a contradiction.
\end{proof}
\begin{rmk}
\label{rmk1b}
A consequence of~\ref{GJ:conv1b} is that  $\nabla u\not=0$ in $\Omega_N\cap\{-1<x<1\}$. Note also that by the previous lemmas it is possible to deduce that in $\Lambda_N\cap\partial\Omega_N$ there are two nondegenerate saddle points.
Indeed, from Theorem~\ref{nodalline} the nodal line is contained in $\Omega_N\cap\{-1<x<1\}$ and~\cite[Lemma 1.2]{l2} tells us that the two points in $\Lambda_N\cap\partial\Omega_N$ are critical points. Moreover, setting $\Lambda_N\cap\partial\Omega_N=\set{q_1,q_2}$ we have $q_1=(o(1),1+o(1))$ and $q_2=(o(1),o(1))$ and then from Proposition~\ref{convcompatti}, writing $q_i:=(x_{q_i},y_{q_i})$, we get for $i=1,2$
\[
\partial_{xx}u_N(q_i)=\partial_{xx}\left(A_0x\sin(\pi y_{q_i})\right)+o(1)=0+o(1)=o(1),
\]
and similarly one has
\begin{align*}
\partial_{xy}u_N(q_i)&=\partial_{xy}\left(A_0x\sin(\pi y_{q_i})\right)+o(1)\\
	&=A_0\pi\cos(\pi y_{q_i})+o(1)=(-1)^iA_0\pi+o(1),\\
\partial_{yy}u_N(q_i)&=\partial_{yy}\left(A_0x_{q_i}\sin(\pi y_{q_i})\right)+o(1)=-A_0\pi^2x_{q_i}\sin(\pi y_{q_i})+o(1)=o(1).
\end{align*}
This yields to
\[
\det\hess_u(q_i)=o(1)-((-1)^iA_0\pi)^2<0,
\]
and the claim follows.
\end{rmk}

\section{The topological argument}
\label{sez3}

Up to the end of this section let us write $u$ instead of $u_N$ for brevity.
Let us recall some notations and some results from~\cite{cc} and~\cite{dgm}.

For every $\theta\in[0,\pi)$ we write $e_{\theta}:=(\cos\theta,\sin\theta)$ and we set
\begin{align*}
u_{\theta}&:=\scal{\nabla u, e_{\theta}}=\frac{\partial u}{\partial e_{\theta}},\\
N_{\theta}&:=\set{p\in\overline{\Omega}_N|u_{\theta}(p)=0}\,\,(\hbox{the nodal set of }u_\theta), \\
M_{\theta}&:=\set{p\in N_{\theta}|\nabla u_{\theta}(p)=\orig}\,\,(\hbox{the singular points of }u_\theta).
\end{align*}
Let us point out that $u_\theta$ clearly solves $-\Delta u_\theta=\lambda_{2,N} u_\theta$ in $\Omega_N$.
Moreover, if the set $\set{u=c}$ is smooth then its curvature is given by
\[
\curv:=-\frac{u_{yy}u_{x}^2-2u_{xy}u_{x}u_{y}+u_{xx}u_{y}^2}{|\nabla u|^3}.
\]
Consider
\[
\Omega_N':=\Set{(x,y)\in\Omega_N|x>1/2}.
\]
In the next proposition we recall some properties of the sets $M_{\theta}$ and $N_{\theta}$ in $\Omega_N'$.
\begin{prop}
\label{propd1}
We have that for every $\theta\in[0,\pi)$,
\begin{enumerate}[(i)]
\item around any $p\in (N_{\theta}\cap\Omega_N')\setminus M_{\theta}$ the nodal set $N_{\theta}$ is a smooth curve;
\item if $p\in M_{\theta}\cap\Omega_N'$, then $N_{\theta}$ consists of at least two smooth curves intersecting transversally at $p$;
\item from the domain monotonicity for Dirichlet eigenvalues there is no nonempty domain $H\subset\Omega_N'$ such that $\partial H\subset N_{\theta}$ (where the boundary of $H$ is considered as a subset of $\R^{2}$);
\item if $p\in\big( N_{\theta}\cap\partial(\Omega_N'\cap\Omega_N)\big)\setminus M_\theta$ by the implicit function theorem one has that around $p$, $N_{\theta}$ is a smooth curve intersecting $\partial\Omega'_N$ transversally  in $p$. 
\end{enumerate}
\end{prop}
\begin{proof}
See~\cite{cc}.
\end{proof}
The following result tells us that for each $\theta\in[0,\pi)$ the nodal sets of $u_\theta$ is a smooth curve without self intersection and every critical point of $u$ is nondegenerate. 
\begin{prop}
\label{prop3b}
For $N$ large enough and for every $\theta\in[0,\pi)$, the nodal set $N_{\theta}$ of the partial derivative $u_\theta$  is a smooth curve in $\overline\Omega_N'$ without self-intersection which hits $\partial\Omega_N'$ exactly at  two  points. Moreover at any critical point of $u$ in $\Omega_N'$ the Hessian matrix has rank $2$.
\end{prop}

\begin{proof}
The proof uses Proposition~\ref{propd1} jointly with Proposition~\ref{convcompatti}.

From the previous points, if we prove that\\
$a)\ M_\theta=\emptyset$ on $N_{\theta}\cap\partial\Omega'_N$,\\
and\\
$b)\ N_{\theta}\cap\partial\Omega_N'=\{p_1,p_2\}$,
\vskip0cm\noindent
we have the claim. Indeed if $a)$ and $b)$ hold then we cannot have self-intersections of $N_\theta$ otherwise $(iii)$ of Proposition~\ref{propd1} fails. So $M_\theta=\emptyset$ and this fact jointly with $(i)$ of Proposition~\ref{propd1} gives the smoothness of $N_\theta$ in $\Omega'_N$. 
In order to prove $a)$ and $b)$ we will show that the following scenario holds:
\begin{itemize}
\item If $\theta$ is far away from $0$ and $\pi$ then $N_\theta$ intersect $\partial\Om'_N$ exactly at \emph{two} points, one of them belonging to $\partial\Om_N$ and the other on the straight line $x=\frac12$.
\item If $\theta$ is close to $0$ and $\pi$ then $N_\theta$ intersect $\partial\Om'_N$ exactly at \emph{two} points, both belonging to the straight line $x=\frac12$.
\item In both cases $N_\theta$ intersect $\partial\Om'_N$ transversely.
\end{itemize}
Now let us consider the two different situations.\vskip0.2cm
\emph{Case $1$: $a)$ and $b)$ hold for $\theta$ far away from $0$ and $\pi$.}\\
From the assumptions on $\Omega_N$ and taking into account that the curvature $\curv$ is positive, there exist $\delta_i:=\delta_i(N)>0$, with $\delta_i\to0$ as $N\to+\infty$, for $i=1,2$, such that for $\theta\in(\delta_1(N),\pi-\delta_2(N))$ there exists a unique $p_1$ on $\partial\Omega_N$ with $x>1/2$ such that the tangent vector of $\partial\Omega_N'$ at $p_1$ is parallel to $e_\theta$.

 It follows that $p_1\in N_\theta$ and from $\curv>0$ we get $p_1\not\in M_\theta$. Indeed
 \[
 u_{\theta\theta}(p_1)=u_{\tangvet\tangvet}(p_1)=\curv(p_1) u_\nu(p_1)\not=0,
 \]
 where $\tangvet$ denotes the unit tangent normal vector, $\nu$ the unit exterior vector and $u_\nu(p_1)\not=0$ by the Hopf boundary lemma. Hence $p\in\big( N_{\theta}\cap\partial(\Omega_N'\cap\Omega_N)\big)\setminus M_\theta$ and $(iv)$ of Proposition~\ref{propd1} implies that $N_{\theta}$ is a smooth curve intersecting $\partial(\Omega_N'\cap\Omega_N)$ transversely  in $p_1$. 
 
Next let us show that for $\theta\in(\delta_1(N),\pi-\delta_2(N))$ and $p=(1/2,y)$ we have that $N_{\theta}$ is a singleton. 
Taking into account~\ref{GJ:conv1b}, one has
\begin{align*}
0&=u_\theta=\cos\theta\partial_xu+\sin\theta\partial_yu\\
	&=\cos\theta\partial_x\left(A_0x\sin(\pi y)\right)+\sin\theta\partial_y\left(A_0x\sin(\pi y)\right)+o(1)\\
	&=A_0\cos\theta\sin\left(\pi y\right)+A_0\frac{\pi}{2}\sin\theta\cos\left(\pi y\right)+o(1),
\end{align*}
if and only if
\[
\cot\theta=-\frac{\pi}{2}\cot(\pi y)(1+o(1)),
\]
which tells us that, for $N$ sufficiently large, there exists exactly one point $p_2=(1/2,y_\theta)$ such that $u_\theta(p_2)=0$. Uniqueness of $p_2$ follows from $\mathcal C^1$ convergence of $u_\theta$  given by Proposition~\ref{convcompatti}.  Moreover similar computations show that $p_2\not\in M_\theta$, indeed 
\begin{align*}
\partial_xu_\theta&=\cos\theta\partial_{xx}u+\sin\theta\partial_{xy}u\\
	&=\cos\theta\partial_{xx}\left(A_0x\sin(\pi y)\right)+\sin\theta\partial_{xy}\left(A_0x\sin(\pi y)\right)+o(1)\\
	&=A_0\pi\sin\theta\cos\left(\pi y\right)+o(1)\not=0,
\end{align*}
for $y\not=1/2+o(1)$. If $y=1/2+o(1)$ one has
\begin{align*}
\partial_yu_\theta&=A_0\pi\cos\theta\cos\left(\pi y\right)-A_0\frac{\pi^2}{2}\sin\theta\sin(\pi y)+o(1)\\
	&=-A_0\frac{\pi^2}{2}\sin\theta+o(1)\not=0.
\end{align*}
So $N_{\theta}\cap\partial\Omega_N'=\{p_1,p_2\}$ and $p_i\not\in M_\theta$ for $i=1,2$; hence $a)$ and $b)$ hold for $\theta\in(\delta_1(N),\pi-\delta_2(N))$.
\vskip0.2cm
\emph{Case $2$: $a)$ and $b)$ hold for $\theta$ close to $0$ and $\pi$.}\\
According to the notations of the previous case let us consider $\theta\in[0,\delta_1(N))\cup(\pi-\delta_2(N),\pi)$. So in this case either $\theta\to0$ or $\theta\to\pi$ as $N\to+\infty$.

Note that here we have that $N_\theta\cap\partial\Omega_N\cap\partial\Omega'_N=\emptyset$ and then we only have to study what happens on the straight line $x=\frac12$. Moreover, Remark~\ref{rmk1b} implies the existence of at least a critical point in $\Omega_N'$ and then $N_\theta\cap\Omega_N'\not=\emptyset$. Since there are no intersections of $N_\theta$ with $\Omega_N\cap\partial\Omega'_N$ then necessarily $N_\theta$ intersects the straight line $x=\frac12$, otherwise $\partial N_\theta$ is a closed curve contained in $\Om'_N$, a contradiction with $iii)$ in Proposition~\ref{propd1}.

Next let us study the intersection of $N_\theta$ with $x=\frac12$.
Recalling that $u(x,y)\sim A_0x\sin(\pi y)$ we get that $u_\theta(1/2,y)=0$ if and only if
\[
0=u_\theta(1/2,y)=A_0\underbrace{\cos\theta}_{\to\pm1}\sin(\pi y)+\frac{A_0}2\underbrace{\sin\theta}_{=o(1)}\cos(\pi y)+o(1),
\]
that implies
\[
\sin\left(\pi y\right)+o(1)=0,
\]
 and hence we have two solutions $y_1=o(1)$ and $y_2=1+o(1)$.
Observe that the last equation admits \emph{exactly} two solution by the $\mathcal C^1$ convergence of $u_\theta$ to $\partial_\theta\left(A_0x\sin(\pi y)\right)$.\\
Finally let us show that both points $p_1=\left(\frac12,y_1\right)$ and  $p_2=\left(\frac12,y_2\right)$ do not belong to $M_\theta$. Indeed, for $N$ large enough
\[
\partial_yu_\theta(p_1)=\frac{A_0}2\pi+o(1)\ne0\quad\hbox{and}\quad\partial_yu_\theta(p_2)=-\frac{A_0}2\pi+o(1)\ne0,
\]
which shows that $p_1,p_2\notin M_\theta$ and as before the implicit function theorem tells us that if $x=1/2$ the nodal set $N_\theta$ is a smooth curve intersecting transversely the line $\{x=1/2\}$ at $p_1$ and $p_2$. This ends the \emph{Case $2$}.
\vskip0.2cm
Hence we proved $a)$ and $b)$ for all $\theta\in[0,\pi)$.\
\vskip0.2cm
Finally at any critical point of $u$ we have that the Hessian matrix is nondegenerate otherwise we deduce that there exists $\theta$ such that $M_\theta\ne\emptyset$ contradicting $a)$.
\end{proof}

For $u$ solution of~\ref{PB1b}, consider the vector field $T:\overline{\Omega_N'}\to\R^{2}$ given by
\[
T(q):=(u_{yy}(q)u_{x}(q)-u_{xy}(q)u_{y}(q),u_{xx}(q)u_{y}(q)-u_{xy}(q)u_{x}(q)),\quad q\in\Omega_N'.
\]
By the smoothness of $u$ we have that $T$ is of class $\mathcal{C}^{1}$. In next lemmas we recall some important properties of the vector field $T$, proved in~\cite{dgm}.

\begin{lemma}[{\cite[Lemma 2]{dgm}}]
\label{lemma1}
If $q\in\Omega_N'$ is such that $T(q)=\orig$ then either
\begin{equation*}
q\hbox{ is a critical point for $u$},
\end{equation*}
or
\begin{equation*}
\det \hess\big(u(q)\big)=0\hbox{ and for $\cos\theta=\frac{u_x(q)}{\sqrt{u_x^2(q)+u_y^2(q)}}$ we have that $q\in M_{\theta}$}.
\end{equation*}
\end{lemma}

From now if $q$ is an isolated zero of $T$, for $r>0$ small enough, we denote by $\mathrm{ind}(T,q):=\deg\big(T,B(q,r),\orig\big)$ where $deg$ denotes the standard Brower degree. 
\begin{lemma}[{\cite[Lemma 3]{dgm}}]
\label{lemma2}
Let $q\in\Omega_N'$ be such that $T(q)=\orig$. Then we have that
\begin{enumerate}[(i)]
\item if  $q$ is a nondegenerate critical point for $u$, then $\mathrm{ind}(T,q)=1$;
\item if $q$ is a singular point belonging to $M_\theta$ for some $\theta\in[0,\pi)$ and it is a nondegenerate critical point for $u_\theta$  then $\mathrm{ind}(T,q)=-1$.
\end{enumerate}
\end{lemma}
Next corollary was proved in~\cite[Corollary 1]{dgm} but we prefer to repeat here the proof.
\begin{cor}[{\cite[Corollary 1]{dgm}}]
\label{cor1}
Let $D\subset\overline{\Omega_N'}$ be such that $M_{\theta}\cap D=\emptyset$ for all $\theta\in[0,\pi)$ and $\orig\not\in T(\partial D)$. If $\deg(D,T,\orig)=1$, then $u$ has exactly one critical point in $D$ which is a maximum with negative definite Hessian.
\end{cor}
\begin{proof}
Since $\orig\not\in T(\partial D)$ the degree of $T$ is well posed. Moreover since $M_{\theta}\cap D=\emptyset$ we have no singular points and moreover all critical points are nondegenerate. So we have finitely many critical points and
\[
1=\deg(D,T,\orig)=\sum_{q\in\set{\text{critical points of }u}}\mathrm{ind}(T,q)=\sharp\set{\hbox{critical points of }u},
\]
which gives the claim.
\end{proof}

Next we prove the uniqueness of critical point in $\Omega_N'$.

\begin{prop}
\label{prop1b}
For $N$ large enough $u_N$ has exactly one critical point in the set $\Omega_N'$. In particular it is a nondegenerate maximum point.
\end{prop}
\begin{proof}
We want to apply Corollary~\ref{cor1}. First of all note that $T\not=\orig$ on $\partial\Omega_N'$. Indeed, in $\partial\Omega_N'\cap\partial\Omega_N$, $T=\orig$ implies
\begin{align*}
-\abs{\nabla u}^3\curv&=u_{yy}u_{x}^2-2u_{xy}u_{x}u_{y}+u_{xx}u_{y}^2\\
&=u_x\left(u_{yy}u_{x}-u_{xy}u_{y}\right)+u_y\left(u_{xx}u_{xy}-u_{xy}u_{x}\right)=0,
\end{align*}
a contradiction with the Hopf boundary lemma and the assumption $\curv>0$ on $\partial\Omega_N$.\\
On the other hand, for $p=(1/2,y)$, using~\ref{GJ:conv1b}, we have
\begin{align}
u_{x}u_{yy}-u_{y}u_{xy}&=\partial_{x}\left(A_0x\sin(\pi y)\right)\partial_{yy}\left(A_0x\sin(\pi y)\right)+\nonumber\\
&\quad-\partial_{y}\left(A_0x\sin(\pi y)\right)\partial_{xy}\left(A_0x\sin(\pi y)\right)+o(1)\nonumber\\
\label{prop2:eq1b}	&=-\frac{A_0^2\pi^2}{2}(1+o(1)),
\end{align}
and then $T\not=\orig$.

So the degree of $T$ is well defined and if for $p_0:=\left(1,\frac12\right)$ the homotopy
\begin{align*}
H:[0,1]\times\overline{\Omega'}_{N}&\to\R^{2}\\
(t,q)&\mapsto tT(q)+(1-t)(q-p_0),
\end{align*}
is admissible then we deduce
\[
\deg(\Omega_{N}',T,\orig)=\deg(\Omega_{N}',I-p_0,\orig)=1,
\]
Assume, by contradiction, that the homotopy $H$ is not admissible. Hence, there exist $\tau\in[0,1]$ and $q:=(x_q,y_q)\in\partial \Omega_{N}'$ such that $H(\tau,q)=\orig$, i.e.
\begin{equation}
\label{sez3:thm1:eq5b}
\begin{cases}
\tau(u_{yy}(q)u_{x}(q)-u_{xy}(q)u_{y}(q))=(\tau-1)(x_{q}-1)\\
\tau(u_{xx}(q)u_{y}(q)-u_{xy}(q)u_{x}(q))=(\tau-1)(y_{q}-1/2).
\end{cases}
\end{equation}
Then, multiplying the first equation by $u_{x}({q})$, the second by $u_{y}({q})$ and summing we get
\begin{equation}
\label{b3b}
-\tau\curv(q)\abs{\nabla u(q)}^3=(\tau-1)[(x_{q}-1)u_{x}(q)+(y_{q}-1/2)u_{y}(q)].
\end{equation}
We want to show that~\ref{b3b} leads to a contradiction. First assume that $q\in\partial\Omega_N'\cap\partial\Omega_N$.

For $(x,y)\in\partial\Omega_N'\cap\partial\Omega_N$ denote by $\nu=(\nu_{x},\nu_{y})$  the unit normal exterior vector at $q$ (consider $\nu$ as the exterior normal to $\partial\Omega_N$ if $x_q=1/2$). Using that $\Omega_{N}'$ is star-shaped with respect to $p_0$ and the Hopf boundary lemma we have
\[
(x_q-1)u_{x}(q)+(y_q-1/2)u_{y}(q)=u_{\nu}(q)[(x_q-1)\nu_{x}+(y_q-1/2)\nu_{y}]<0.
\]
Since $\curv>0$ on $\partial\Omega_N'\cap\partial\Omega_N$, from~\ref{b3b} we get a contradiction. 
It follows that $q\not\in\partial\Omega_N'\cap\partial\Omega_N$ and then $q=(1/2,y_q)$.
 From~\ref{prop2:eq1b} and the first line of~\ref{sez3:thm1:eq5b} we get
\[
-\frac{A_0^2\pi^2}{2}\tau(1+o(1))=(\tau-1)(1/2-1)=\frac{1-\tau}{2},
\]
again a contradiction.

So $\deg(\Omega_{N}',T,\orig)=1$  and by Corollary~\ref{cor1} we get that there exists exactly one critical point in $\Omega_{N}'$: a maximum with negative definite Hessian.
\end{proof}

Similarly we can prove the following proposition.

\begin{prop}
\label{prop2b}
For $N$ big enough, $u_N$ has exactly one critical point in the set $\{(x,y)\in\Omega_N|x<-1/2\}$. In particular, it is a nondegenerate minimum point.
\end{prop}

Finally the proof of Theorem~\ref{thm1b} easily follows.

\begin{proof}[Proof of Theorem~\ref{thm1b}]
The proof follows from Remark~\ref{rmk1b}, Proposition~\ref{prop1b} and Proposition~\ref{prop2b}. Observe that by the local convergence of $u_N$ to $u_\infty(x,y)=A_0x\sin(\pi y)$ we get that  $|P_N|,|Q_N|\to+\infty$.
\end{proof}

\section{Convex perturbations of rectangles: proof of Theorem~\ref{thm1}}
\label{sez4}

We start recalling the asymptotic expansion of $u_{N,m}$ given in~\cite{gj4}.

\begin{thm}[{\cite[Theorem 1]{gj4}}]
\label{thmGJPacific}
There is a number $a:=a(\phi)\in[0,\max\phi]$ such that for each $m\in\N$ the $m$-th Dirichlet eigenvalue of $\mathcal R_N$ (see \ref{rett}) satisfies
\[
\lambda_{m,N}=\pi^2+\frac{m^2\pi^2}{(N+a(\phi))^2}+O(N^{-5}),\quad N\to\infty.
\]
In particular, the eigenvalues $\lambda_{1,N},\dots,\lambda_{m,N}$ of $\mathcal R_N$ are simple for $N$ sufficiently large. The suitably rescaled eigenfunction $u_{m,N}$ satisfies, for all multiindices $\alpha$,
\begin{equation}
\label{GJ:conv1}
\sup_{\substack{x>3\log N\\0<y<1}}\left| D^\alpha\left(u_{m,N}(x,y)-v_m(x,y)\right)\right|=O(N^{-3}),
\end{equation}
where
\[
v_m(x,y):=\sin\left(m\pi\frac{x+a(\phi)}{N+a(\phi)}\right)\sin\left(\pi y\right),
\]
and 
\[
\sup_{\substack{x\le3\log N\\0<y<1}}\left| u_{m,N}(x,y)\right|=O(N^{-1}\log N).
\]
\end{thm}

We prove Theorem~\ref{thm1} for $m=2$, the general case is a simple generalization as will be clear from the proof, see also Remark~\ref{rmk:alessandrini}. We write $u_N=u_{2,N}$ and $v=v_2$ for brevity.

For future convenience let us set
\begin{equation*}
\label{def_xN}
\begin{split}
x_N&:=\frac{1}{2}(N+a)-a,\\
x_N^+&:=\frac{1}{4}(N+a)-a,\\
x_N^-&:=\frac{3}{4}(N+a)-a,\\
 x_N'&:=\frac{1}{12}(N+a)-a.
\end{split}
\end{equation*}
\begin{prop}
\label{prop1}
For $N$ big enough, the eigenfunction $u_N$ has exactly one nondegenerate maximum point and one nondegenerate minimum point in the set $\mathcal R_N\cap\{x>3\log N\}$. 
\end{prop}
\begin{proof}
From~\ref{GJ:conv1} easily follows that $u_N$ has a maximum point close to $(x_N^+,1/2)$ and a minimum point close to $(x_N^-,1/2)$. To show that they are the only ones and are nondegenerate, let $p:=(x_p,y_p)\in \mathcal R_N\cap\{x>3\log N\}$ be a critical point for $u_N$.

Then~\ref{GJ:conv1} implies that there exist a continuous and decreasing function $h:(0,+\infty)\to(0,+\infty)$ such that $\lim_{N\to+\infty}h(N)=0$ and one of the following occurs
\begin{align}
\label{p1}p&\in B_{h(N)}(x_N^+,1/2),\\
\label{p2}p&\in B_{h(N)}(x_N^-,1/2),\\
\label{p3}p&\in B_{h(N)}(x_N,0)\cap\Omega_N,\\
\label{p4}p&\in B_{h(N)}(x_N,1)\cap\Omega_N,\\
\label{p5}p&\in B_{h(N)}(N,0)\cap\Omega_N,\\
\label{p6}p&\in B_{h(N)}(N,1)\cap\Omega_N.
\end{align}
Assume~\ref{p1}, then from~\ref{GJ:conv1} one has
\begin{align*}
\partial_{xx}u_N(p)&=\partial_{xx}v(p)+O(N^{-3})\\
&=-\frac{4\pi^2}{(N+a)^2}\sin(\pi/2)\sin(\pi/2)(1+o(1))=-\frac{4\pi^2}{(N+a)^2}(1+o(1))
\end{align*}
and similarly
\[
\partial_{xy}u_N(p)=o(N^{-1})\quad\text{and}\quad \partial_{yy}u_N(p)=-\pi^2(1+o(1)).
\]
Hence $p$ is a nondegenerate maximum point. Moreover, we can find $r>0$ independent from $N$ such that the following homotopy $H:[0,1]\times\overline{B_{r}(x_N^+,1/2)}\ \to\R^{2}$
\[
H(t,q)=t\nabla u_N(q)+(1-t)\nabla v(q),
\]
is admissible  for $N$ big enough. Then
\[
\deg(B_{r}(x_N^+,1/2),\nabla u_N,\orig)=\deg(B_{r}(x_N^+,1/2),\nabla v,\orig)=1,
\]
shows that there is exactly one critical point satisfying~\ref{p1}. If we assume~\ref{p2}, by similar computations, we obtain the existence of exactly one nondegenerate minimum point in $B_{h(N)}(x_N^-,1/2)$.

Now assume~\ref{p3} i.e. $p\in B_{h(N)}(x_N,0)\cap \mathcal R_N$. Then the same computation as before tell us that $p$ is a nondegenerate saddle point, indeed one has
\begin{equation}
\label{saddle}
\partial_{xx}u_N(p)=o(N^{-2}),\!\quad \partial_{xy}u_N(p)=-\frac{2\pi^2}{N+a}(1+o(1)),\!\quad \partial_{yy}u_N(p)=o(1).
\end{equation}
Now, if $\Lambda_N:=\overline{\{(x,y)\in \mathcal R_N|u_N(x,y)=0\}}$ is the nodal line of $u_N$, let $p_N:=(\tilde x_N,0)\in\partial \mathcal R_N\cap\Lambda_N$. Since $\mathcal R_N$ is convex we know from~\cite[Theorem 1]{a} that $\Lambda_N$ intersects $\partial \mathcal R_N$ transversally at $p_N$. In particular $\partial_yu_N(p_N)=0$ and then $p_N$ is a critical point for $u$ and~\ref{saddle} shows that it is a nondegenerate saddle point. Since both $p$ and $p_N$ are nondegenerate we can find $g(N)\in(0,h(N))$ such that $p\in B_{h(N)}(x_N,0)\setminus\overline{B_{g(N)}(x_N,0)}$, and for $r>0$ suitably small and $N$ big enough, since in every critical point in $\omega_N:=B_{r}(x_N,0)\setminus\overline{B_{g(N)}(x_N,0)}\cap\Omega_N$ one has
\[
\det\hess u_N=-\left(\frac{2\pi^2}{N+a}\right)^2(1+o(1))<0,
\]
thanks to~\ref{saddle}, and since at least $p$ belongs to $\omega_N$ it follows $\deg(\omega_N,\nabla u_N,\orig)\le-1$ and then
\[
-1\ge\deg(\omega_N,\nabla u_N,\orig)=\deg(\omega_N,\nabla v,\orig)=0,
\]
a contradiction.

The same argument shows that~\ref{p4},~\ref{p5} and~\ref{p6} cannot occur and the proof is complete.
\end{proof}

\begin{rmk}
\label{rmk:alessandrini}
In case $m>2$,~\cite[Theorem 1]{a} still ensures that the nodal line intersects the boundary $\partial \mathcal R_N$ transversally at $2m$ different points 
\end{rmk}

\begin{prop}
\label{prop2}
For $N$ big enough, $u_N$ has no critical point in the set
\[
\mathcal R_N':=\Set{(x,y)\in \mathcal R_N|x< x_N'}.
\]
\end{prop}
\begin{proof}
Let us point out that, from the estimate~\ref{GJ:conv1} and since $x_N'<x_N$, it follows $u_N>0$ in $\mathcal R_N'$. By the domain monotonicity for Dirichlet eigenvalues one has $\lambda_1(\mathcal R_N')>\lambda_{2,N}$ and then the operator $-\Delta-\lambda_{2,N}$ satisfies the maximum principle in $\mathcal R_N'$. From~\ref{GJ:conv1} one has for all $y\in(0,1)$
\[
\partial_xu_N(x_N',y)=\frac{2\pi}{N+a}\cos(\pi/6)\sin\left(\pi y\right)(1+o(1))\ge0.
\]
Therefore, $\partial_xu_N\ge0$ on $\partial \mathcal R_N'$ and then the maximum principle gives $\partial_xu_N>0$ on $\mathcal R_N'$.
\end{proof}

\begin{proof}[Proof of Theorem~\ref{thm1}]
The proof is an obvious consequence of Proposition~\ref{prop1} and Proposition~\ref{prop2}.
\end{proof}

\bibliographystyle{alphaabbrv}
\bibliography{NuovaVersione.bib}

\end{document}